\documentclass[11pt]{article}
\usepackage{graphicx}
\usepackage{tabularx}
\usepackage{url}
\usepackage{amsthm}

\usepackage{color}
\usepackage{xspace}
\usepackage{tikz}
\usetikzlibrary{decorations,decorations.pathmorphing,decorations.pathreplacing,fit,positioning,arrows}

\tikzstyle{w_vertex}=[circle,fill=black!100,text=white,inner sep=0.4mm,draw]
\tikzstyle{vertex}=[circle,fill=black!100,text=white,inner sep=0.8mm]
\tikzstyle{point}=[circle,fill=black,inner sep=0.1mm]

\theoremstyle{plain}
\newtheorem{theorem}{Theorem}
\newtheorem{lemma}{Lemma}

\newtheorem{claim}{Claim}

\theoremstyle{definition}

\newtheorem{problem}{Problem}

\theoremstyle{remark}

\date{}

\title{Tree-width dichotomy}

\author{Vadim Lozin\thanks{Mathematics Institute, University of Warwick, Coventry, CV4 7AL, UK. Email:  V.Lozin@warwick.ac.uk} \and 
Igor Razgon \thanks{Department of Computer Science and Information Systems Birkbeck University of London. Email: igor@dcs.bbk.ac.uk}}

\usepackage{caption}
\usepackage{amsmath} 
\usepackage{amsfonts}
\usepackage[cp1251]{inputenc}
\usepackage{colordvi}
\usepackage{graphicx}
\usepackage{tikz}
\usepackage{subcaption}
\usetikzlibrary{decorations.pathreplacing,arrows,automata}
\begin{document}
\maketitle

\newtheorem{obs}{Observation}

\newtheorem{prop}{Proposition}
\newtheorem{cor}{Corollary}

\def\N{\mathbb{N}}
\def\t{\sim}
\def\nt{\nsim}
\def\1{n+1}
\def\2{n+2}

\begin{abstract}
We prove that the tree-width of graphs in a hereditary class defined by a finite set $F$ of forbidden induced subgraphs is bounded 
if and only if $F$ includes a complete graph, a complete bipartite graph, a tripod (a forest in which every connected component 
has at most 3 leaves) and the line graph of a tripod.  
\end{abstract}



\section{Introduction}


A class of graphs, also known as a graph property, is a set of graphs closed under isomorphism.
A property is {\it hereditary} if it is closed under taking induced subgraphs. 
The world of hereditary properties is rich and diverse, and it contains various classes of theoretical or practical importance, such as perfect graphs, interval graphs,
permutation graphs, bipartite graphs, planar graphs, threshold graphs, split graphs, graphs of bounded vertex degree, graphs of bounded tree-width, etc. 
It also contains all classes closed under taking subgraphs, minors, induced minors, vertex-minors, etc.

Tree-width is a graph parameter, which is important in algorithmic graph theory, because many problems that are NP-hard for general graphs
become polynomial-time solvable for graphs of bounded tree-width. Many hereditary classes have been shown to be of bounded tree-width, 
such as outerplanar graphs or graphs of bounded vertex degree that also have bounded chordality. For many other classes, 
tree-width has been shown to be unbounded, which includes planar graphs or graphs of vertex degree at most 3. In the present paper, 
we characterize the family of hereditary classes of bounded tree-width by means of four critical properties. To better explain our approach, 
let us illustrate it with the following example. 

The bottom of the hierarchy of hereditary properties consists of classes containing finitely many graphs,
i.e., classes where the number of vertices is bounded. According to Ramsey's Theorem,
there exist precisely two minimal hereditary classes that are infinite: complete graphs and edgeless graphs. 
In other words, in the universe of hereditary classes the complete graphs and the edgeless graphs are minimal obstructions 
(minimal forbidden elements) for the family of classes of bounded vertex number.

The idea of minimal obstructions for {\it families of classes} generalizes 
the idea of minimal forbidden induced subgraphs for individual hereditary classes. This idea always works in a well-quasi-ordered world.
For instance, the graph minor relation is known to be a well-quasi-order \cite{minor-wqo},
and in the family of minor-closed classes of graphs the planar graphs constitute a unique minimal obstruction for classes of bounded tree-width \cite{planar}.
However, the induced subgraph relation is not a well-quasi-order and in the universe of hereditary properties minimal classes outside of a particular family may not exist, 
in  which case we employ the notion of boundary classes. This is a relaxation of the notion of minimal classes and it is defined as follows.

Let $\cal A$ be a family of hereditary classes closed under taking subclasses, and let $X_1\supseteq X_2\supseteq X_3\supseteq\ldots$
be a chain of classes not in $\cal A$. The intersection of all classes in this chain is called a {\it limit} class and a minimal limit class 
is called a {\it boundary} class for $\cal A$. The importance of this notion is due to the fact that it plays the role of minimal obstructions for hereditary classes 
defined by finitely many forbidden induced subgraphs (finitely defined classes, for short). In other words, a finitely defined class $X$ belongs to $\cal A$ if and only if
$X$ contains none of the boundary classes for $\cal A$.

The notion of boundary classes was formally introduced by Alekseev in \cite{Alekseev}, where he proved basic results related to this notion. 
However, implicitly this notion appeared earlier. For instance, if $\cal A$ is the family of hereditary classes of bounded chromatic number,
then in the terminology of limit and boundary classes the famous result of Erd\H{o}s \cite{Erdos} states that 
the class of forests is a limit class for $\cal A$, while the {G}y\'{a}rf\'{a}s-{S}umner conjecture \cite{Sumner}
claims that this is a minimal limit, i.e. boundary, class for $\cal A$. Notice that for this family there is one more obstruction, the class of complete 
graphs, which is a minimal hereditary class of unbounded chromatic number.

Together, minimal and boundary classes are known as critical properties. In the present paper, we show that for the family of hereditary classes of bounded tree-width
there are exactly four critical properties: two minimal classes (complete graphs and complete bipartite graphs) and two boundary classes (tripods, i.e. 
forests in which every connected component has at most 3 leaves, and the line graphs of tripods). 
In other words, we show that the tree-width of graphs in a finitely defined class $X$ is bounded if and only if $X$ excludes a graph from each of the four critical classes. 
The `only if' part follows from known results that we report in Section~\ref{sec:pre} together with some other information related to the topic of the paper. 
The `if' part is proved in Section~\ref{sec:bounded}.
Section~\ref{sec:discussion} concludes the paper with a discussion.


\section{Preliminaries}
\label{sec:pre}

All graphs in this paper are simple, i.e. finite, undirected, without loops and multiple edges. 
The vertex set and the edge set of a graph $G$ are denoted $V(G)$ and $E(G)$, respectively. 
By $K_n$ we denote the complete graph on $n$ vertices and by $K_{n,m}$ the complete bipartite graph with parts of size $n$ and $m$.

A graph $H$ is an induced subgraph of a graph $G$ if $H$ can be obtained from $G$ by deleting some vertices.
The induced subgraph of $G$ obtained by deleting the vertices outside of a set $U\subseteq V(G)$ is denoted $G[U]$.
If $G$ does not contain $H$ as an induced subgraph, then we say that $G$ is $H$-free.

A subdivision of an edge is the operation of introducing a new vertex on the edge. A graph $H$ is a subdivision 
of a graph $G$ if $H$ is obtained from $G$ by subdividing its edges. $H$ is a $(\le k)$-subdivision of $G$
if $H$ is obtained from $G$ by subdividing each edge of $G$ at most $k$ times.
$H$ is a proper subdivision of $G$ if $H$ is obtained from $G$ by subdividing each edge of $G$ at least once.

A clique in a graph is a subset of pairwise adjacent vertices and an independent set is a subset of pairwise non-adjacent vertices.
The Ramsey number $R(p,q)$ is the minimum $n$ such that every graph with at least $n$ vertices has either an independent set of size $p$ 
or a clique of size $q$.  


\subsection{Tree decomposition}
\label{sec:tree-width}

Let $G$ be a graph,  $T$ be a tree, and  ${\cal V}=(V_t)_{t\in T}$ be a family of 
vertex sets $V_t\subseteq V(G)$, called {\it bags}, indexed by the nodes $t$ of $T$. 
The pair $(T,\cal V)$ is called a {\it tree decomposition} of $G$ if  
\begin{itemize}
\item $V(G)=\cup_{t\in T}V_t$,
\item for every edge $e$ of $G$, there exists a bag $V_t$ containing both endpoints of $e$,
\item for any three nodes $t_1,t_2,t_3$ of $T$ such that $t_2$ lies on the unique path, 
connecting $t_1$ to $t_3$ in $T$, we have $V_{t_1}\cap V_{t_3}\subseteq V_{t_2}$. 
\end{itemize}
The {\it width} of a tree decomposition is the size of its largest bag minus one. The {\it tree-width} of 
$G$ is the minimum width among all possible tree decompositions of $G$. 

Let $(T,\cal V)$ be a tree decomposition of $G$ and $xy$ an edge in $T$.
We denote by $T_x$ and $T_y$ the connected components of $T-xy$ containing $x$ and $y$, respectively.
Also, let $U_x$ be the union of the bags corresponding to the nodes of $T_x$, and 
let $U_y$ be the union of the bags corresponding to the nodes of $T_y$. 
The set $Z_{xy}:=U_x\cap U_y=V_x\cap V_y$ is a separator in $G$, i.e. there are no edges between 
$U_x-Z$ and $U_y-Z$. The decomposition ${\cal V}=(V_t)_{t\in T}$ is {\it tight} if 
for any edge $xy\in E(T)$ and any two vertices $u,v\in Z_{xy}$, both $G[U_x]$ and $G[U_y]$ 
contain a $u$-$v$-path with no internal vertices in $Z_{xy}$.

For a node $x$ of $T$, the {\it torso} at $x$ is the graph obtained from $G[V_x]$ by creating a clique in $V_x\cap V_y$ 
for every edge $xy\in E(T)$.

A $b$-{\it block} in a graph $G$ is a maximal subset of at least $b$ vertices no two of which can be separated in $G$
by deleting fewer than $b$ vertices.  The maximum $b$ such that $G$ contains a $b$-block is the {\it block number} of $G$.

\begin{theorem}\label{thm:block}{\rm \cite{block}}
Let $k\ge 3$ be a positive integer and $G$ a graph. If $G$ has no $k$-block, then there is a tight tree-decomposition $(T,\cal V)$
of $G$ such that every torso has fewer than $k$ vertices of degree at least $2(k-1)(k-2)$.
\end{theorem}


\subsection{Critical classes for bounded tree-width}

It is well-known that complete graphs and complete bipartite graphs can have arbitrarily large tree-width
and it is not difficult to see that these are {\it minimal} hereditary classes of unbounded tree-width.
Moreover, tree-width remains unbounded even if we forbid a complete graph $K_p$ and a complete bipartite graph $K_{s,s}$
as an induced subgraph, because for any $p\ge 3$ and $s\ge 4$ this includes all graphs of vertex degree at most $3$,
which are known to be of unbounded tree-width. We observe that by forbidding $K_p$ and $K_{s,s}$ as induced subgraphs 
we forbid $K_{t,t}$ as a subgraph for any $t\ge R(s,p)$. We will refer to these graphs as graphs with no    
$K_{t,t}$-subgraph, or simply as graphs of bounded biclique number.

In the universe of hereditary classes of graphs of bounded biclique number minimal classes of unbounded tree-width do not exist
and hence we look for boundary classes that are not minimal. Two of them have been identified in \cite{boundary-CPC}. These are 
\begin{itemize}
\item[$\cal S$] the class of graphs every connected component of which has the form $S_{i,j,k}$ represented in Figure~\ref{fig:ST} (left),
\item[$\cal T$] the class of graphs every connected component of which has the form $T_{i,j,k}$ represented in Figure~\ref{fig:ST} (right).
\end{itemize}

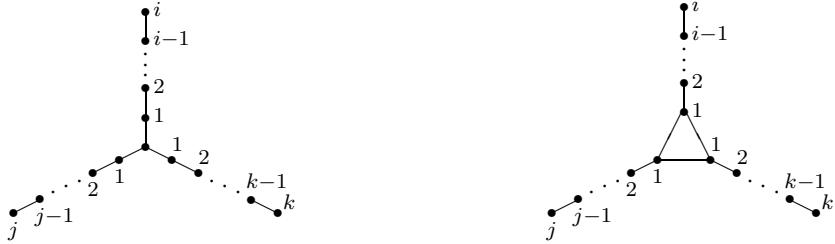
\begin{figure}[ht]
\begin{center} \begin{picture}(200,80)
\put(110,35){\circle*{3}}
\put(110,46){\circle*{3}}
\put(110,57){\circle*{3}}
\put(110,75){\circle*{3}}
\put(110,86){\circle*{3}}
\put(110,62){\circle*{1}}
\put(110,66){\circle*{1}}
\put(110,70){\circle*{1}}
\put(110,35){\line(0,1){11}}
\put(110,46){\line(0,1){11}}
\put(110,75){\line(0,1){11}}
\put(100,30){\circle*{3}}
\put(90,25){\circle*{3}}
\put(70,15){\circle*{3}}
\put(60,10){\circle*{3}}
\put(85,22){\circle*{1}}
\put(80,20){\circle*{1}}
\put(75,18){\circle*{1}}
\put(110,35){\line(-2,-1){10}}
\put(100,30){\line(-2,-1){10}}
\put(70,15){\line(-2,-1){10}}
\put(120,30){\circle*{3}}
\put(130,25){\circle*{3}}
\put(150,15){\circle*{3}}
\put(160,10){\circle*{3}}
\put(135,22){\circle*{1}}
\put(140,20){\circle*{1}}
\put(145,18){\circle*{1}}
\put(110,35){\line(2,-1){10}}
\put(120,30){\line(2,-1){10}}
\put(150,15){\line(2,-1){10}}
\put(113,46){$_1$}
\put(113,57){$_2$}
\put(113,75){$_{i-1}$}
\put(113,86){$_i$}
\put(98,23){$_1$}
\put(88,18){$_2$}
\put(68,8){$_{j-1}$}
\put(58,3){$_j$}
\put(120,35){$_1$}
\put(130,30){$_2$}
\put(148,21){$_{k-1}$}
\put(162,13){$_k$}
\end{picture}
\begin{picture}(200,80)
\put(110,48){\circle*{3}}
\put(110,59){\circle*{3}}
\put(110,77){\circle*{3}}
\put(110,88){\circle*{3}}
\put(110,64){\circle*{1}}
\put(110,68){\circle*{1}}
\put(110,72){\circle*{1}}
\put(120,30){\line(-1,2){9}}
\put(110,48){\line(0,1){11}}
\put(110,77){\line(0,1){11}}
\put(100,30){\circle*{3}}
\put(90,25){\circle*{3}}
\put(70,15){\circle*{3}}
\put(60,10){\circle*{3}}
\put(85,22){\circle*{1}}
\put(80,20){\circle*{1}}
\put(75,18){\circle*{1}}
\put(100,30){\line(1,0){20}}
\put(100,30){\line(-2,-1){10}}
\put(70,15){\line(-2,-1){10}}
\put(120,30){\circle*{3}}
\put(130,25){\circle*{3}}
\put(150,15){\circle*{3}}
\put(160,10){\circle*{3}}
\put(135,22){\circle*{1}}
\put(140,20){\circle*{1}}
\put(145,18){\circle*{1}}
\put(100,30){\line(1,2){9}}
\put(120,30){\line(2,-1){10}}
\put(150,15){\line(2,-1){10}}
\put(113,48){$_1$}
\put(113,59){$_2$}
\put(113,77){$_{i-1}$}
\put(113,88){$_i$}
\put(98,23){$_1$}
\put(88,18){$_2$}
\put(68,8){$_{j-1}$}
\put(58,3){$_j$}
\put(120,35){$_1$}
\put(130,30){$_2$}
\put(148,21){$_{k-1}$}
\put(162,13){$_k$}
\end{picture} 
\end{center}
\caption{The graphs $S_{i,j,k}$ (left) and  $T_{i,j,k}$ (right) }
\label{fig:ST}
\end{figure}

We call graphs in $\cal S$ the {\it tripods}. It is not difficult to see that graphs in $\cal T$ are the line graphs of tripods.
We repeat that boundary classes play the role of minimal obstructions for classes defined by finitely many forbidden induced subgraphs, 
which implies the following conclusion.

\begin{lemma}\label{lem:boundary}{\rm \cite{boundary-CPC}}
If $X$ is a finitely defined class that contains $\cal S$ or $\cal T$, then the tree-width of graphs in $X$ is unbounded.  
\end{lemma}

In \cite{boundary-CPC}, it was shown that $\cal S$ and $\cal T$ are the {\it only} boundary classes in the universe of planar graphs,
i.e. in any hereditary class of planar graphs excluding a graph from $\cal S$ and a graph from $\cal T$ the tree-width is bounded.  
Also, in \cite{degree}, it was shown that the same is true for the universe of hereditary classes of bounded vertex degree.

\begin{theorem}\label{thm:degree}{\rm \cite{degree}}
For every positive integer $d$, every graph $G_1\in \cal S$ and every graph $G_2\in \cal T$, there is 
a positive integer $t=t(d,G_1,G_2)$ such that all $(G_1,G_2)$-free graphs of vertex degree at most $d$ 
have tree-width at most $t$.
\end{theorem}

This result for bounded degree graphs was obtained by reducing it to graphs of bounded chordality (the length of a longest chordless cycle), 
because bounded chordality together with bounded degree imply bounded tree-width \cite{chordality}. In \cite{absence}, the result for 
bounded chordality was extended from graphs of bounded degree to graphs of bounded biclique number by means of Theorem~\ref{thm:block}.
In the next section, we use a similar strategy to extend Theorem~\ref{thm:degree} from graphs of bounded degree to graphs of bounded biclique number.  


\section{Main result}
\label{sec:bounded}


We start with the following helpful lemma, which will be used repeatedly throughout the proof.

\begin{lemma}\label{lem:clique}
For all positive integers $a$ and $b$, there is a positive integer $C=C(a,b)$ such that if 
a graph $G$ contains a collection of $C$ pairwise disjoint subsets of vertices, each of size at most $a$ 
and with at least one edge between any two of them, then $G$ contains a $K_{b,b}$-subgraph.   
\end{lemma}

\begin{proof}
Let $P(r,m)$ denote the minimum $n$ such that in every
coloring of the elements of an $n$-set with $r$ colors there exists a subset of $m$ elements
of the same color (the pigeonhole principle).

We define $r:= P(a^b,b)$ and $C(a,b):= 2P(a^r,b)$ and assume $G$ contains a family of 
$C$ pairwise disjoint subsets of vertices, each of size at most $a$ and with at least one 
edge between any two of them. We split this family arbitrarily into two subfamilies  $\cal A$ and $\cal B$,
each of size $P(a^r,b)$ and consider an arbitrary
collection $A$ of $r$ sets from $\cal A$. Since each set in $\cal B$ has a neighbor in each set in $A$,
the family of the sets in $\cal B$ can be colored with at most $a^r$ colors so that all sets of
the same color have a common neighbor in each of the $r$ chosen sets of collection $A$. 
Since $|{\cal B}|=P(a^r,b)$, one of the color classes contains a collection $B$ of at
least $b$ sets. For each set in $A$, we choose a vertex which is a common neighbor for
all sets in $B$ and denote the set of $r$ chosen vertices by $U$. The vertices of $U$ can be
colored with at most $a^b$ colors so that all vertices of the same color have a common
neighbor in each of the $b$ sets of collection $B$. By the choice of $r$, $U$ contains a color
class $U_1$ of least $b$ vertices. For each set in $B$, we choose a vertex which is a common
neighbor for all vertices of $U_1$ and denote the set of $b$ chosen vertices by $U_2$. Then
$U_1$ and $U_2$ form a biclique $K_{b,b}$. 
\end{proof}

Before we proceed to graphs without tripods, we prove the following key result, which is of independent interest.  

\begin{theorem}  
\label{thm:bigclique}
For all positive integers $r$ and $p$, there is a positive integer $m=m(r,p)$
such that every graph $G$ containing a $(\le p)$-subdivision of $K_m$ as a subgraph 
contains either $K_{p,p}$ as a subgraph or a proper $(\leq p)$-subdivision of $K_{r,r}$ as an induced subgraph. 
\end{theorem}
 
\begin{proof}
For integers $a,b>0$ and $i \geq 0$, we recursively define $R^i(a,b)$ as follows: $R^0(a,b)=R(a,b)$ and for $i>0$, $R^i(a,b)=R(R^{i-1}(a,b),b)$. 
With $C(a,b)$ coming from Lemma~\ref{lem:clique}, we also define 
\begin{itemize}
\item[] $q=q(r,p)=R(r,C(rp,p))$,
\item[] $c=c(r,p)=R^{q}(r,C(p,p))$, 
\item[] $d=d(r,p)=c(r,p)+q(r,p)$, 
\item[] $m=m(r,p)=R(d,2p)$.
\end{itemize}
Suppose that $G$ contains a $(\leq p)$-subdivision of $K_m$ as a subgraph. For any two vertices $u,v$ of the $K_m$ we denote by $P_{u,v}$ the path obtained by subdividing the edge $uv$ 
and assume without loss of generality that this path is induced, i.e. chordless. 
By definition of $m$, either $2p$ vertices of the $K_m$ induce a clique, in which case $G$ contains a $K_{p,p}$-subgraph,
or $d$ vertices of the $K_m$ form an independent set $A$ in $G$, in which case $G$ contains a proper $(\leq p)$-subdivision of $K_d$ as a subgraph.

We partition the vertices of $A$ arbitrarily into two subsets $B$ of size $q=q(r,p)$ and $C$ of size $c=c(r,p)$. 
Also, let $b_1, \ldots, b_{q}$ be an arbitrary ordering of the vertices in $B$ and let $B_i=\{b_1, \ldots, b_i\}$, $B_0=\emptyset$.

\begin{claim}
If $G$ does not contain a $K_{p,p}$-subgraph, then there are subsets $C_{0}, \ldots, C_{q}$
such that for each $i$
\begin{itemize}
\item[(a)] $C_i$ is of size at least $R^{q-i}(r,C(p,p))$,
\item[(b)] for each vertex $u \in B_i$ and every two distinct vertices
$v_1,v_2 \in C_i$ there are no edges between $P_{u,v_1} - u$ and $P_{u,v_2} - u$. 
\end{itemize} 
\end{claim}

\begin{proof}
We prove the claim by induction on $i$. For $i=0$, the claim follows trivially with $C_0=C$. 
Now let  $i>0$. Consider the graph with vertex set $C_{i-1}$ in which two vertices $v_1,v_2$ are adjacent 
if and only if there is an edge between $P_{u_i,v_1} - u_i$ and $P_{u_i,v_2} - u_i$. 
Then by the inductive assumption, this graph has  either a clique of size $C(p,p)$, in which case $G$ contains a $K_{p,p}$-subgraph by Lemma~\ref{lem:clique},
or an independent set $C_i$ of size $R^{q-i}(r,C(p,p))$ and the claim follows. 
\end{proof}

Assume that $G$ does not contain a $K_{p,p}$-subgraph. According to the above claim, there is a set $C_q$ satisfying (a) and (b). 
Let $V$ be an arbitrary subset of $C_{q}$ of size $r$ and let $U=B_{q}$. 
For each $u \in U$, let $S(u)$ be the set consisting of all the internal vertices of the $P_{u,v}$ paths over all $v \in V$.

Consider the graph with vertex set $U$ in which two vertices $u_1,u_2$ are adjacent 
if and only if there is an edge between $S(u_1)$ and $S(u_2)$. 
Since $|U|=R(r,C(rp,p))$, this graph either has a clique of size $C(rp,p)$, in which case $G$ contains a $K_{p,p}$-subgraph by Lemma~\ref{lem:clique},
or an independent set $W$ of size $r$.
In the latter case, $W$ and $V$ together with the vertices of the paths $P_{u,v}$ 
for all $u\in W$ and $v\in V$ induce a proper $(\leq p)$-subdivision of $K_{r,r}$.
\end{proof}

Now we turn to graphs without tripods. 
For ease of notation, we denote a tripod  $S_{p,p,p}$ simply by $S_p$. Similarly, we denote $T_{p,p,p}$ by $T_p$.

\begin{lemma}\label{lem:block}
For every positive integers $p$, there is a positive integer $b=b(p)$ such that all $S_p$-free graphs with no $K_{p,p}$-subgraph
have block number at most $b$.
\end{lemma}

\begin{proof}
We denote by $C=C(p,p)$ the constant from Lemma~\ref{lem:clique} and $m=m(p,p)$ the constant from Theorem~\ref{thm:bigclique}.
Finally, let $b={\binom m 2}p+(m-2)+R(3,C)$. 

\begin{claim}\label{claim:1}
If $G$ contains a $b$-block, then it contains a $(\le p)$-subdivision of $K_m$ as a subgraph.  
\end{claim}
\begin{proof}
Let $B$ be a $b$-block in $G$ and consider two non-adjacent vertices $x,y\in B$. Since $x$ cannot be separated from $y$ 
by deleting fewer than $b$ vertices, there must exist a collection ${\cal P}(x,y)$ of at least $b$ internally disjoint paths connecting $x$ to $y$ by Menger's Theorem.
Without loss of generality, we assume that each path in ${\cal P}(x,y)$ is chordless. 

Assume $c:=R(3,C)$ of the paths in ${\cal P}(x,y)$ have length more than $p$. Then we may consider the first $p$ vertices (different from $x$) in each of these long paths, 
creating in this way a collection of $R(3,C)$ pairwise disjoint subsets, each of size $p$. Then either three of the subsets have no edges between them, 
in which case together with $x$ they induce an $S_p$,
or $C$ of the subsets have an edge between any two of them, in which case a $K_{p,p}$-subgraph, arises by Lemma~\ref{lem:clique}.  
A contradiction in both cases shows that fewer than $c$ paths in ${\cal P}(x,y)$ have length more than $p$.

For any two non-adjacent vertices $x,y\in B$, the collection ${\cal P}(x,y)$ contains a sub-collection ${\cal P}'(x,y)$ of at least $b-(m-2)-c\ge {\binom m 2}p$ paths of length at most $p$  with no internal vertices in $B$.  
Therefore, for each pair $x,y$ of non-adjacent vertices in $B$ one can choose a path from ${\cal P}'(x,y)$ so that the chosen paths intersect only at the endpoints, i.e. they create the desired subdivision of $K_m$.
This can be easily seen by induction on the number of pairs: if the paths have been chosen for $i-1<{\binom m 2}p$ pairs,
then for the $i$-th pair $x,y\in B$ the set ${\cal P}'(x,y)$ contains at most $(i-1)p$ paths sharing internal vertices with the previously chosen paths, and hence ${\cal P}'(x,y)$ contains a path   
that intersects the previously chosen paths only at the endpoints. 
\end{proof}

According to this claim and Theorem~\ref{thm:bigclique}, if  $G$ contains a $b$-block, then it contains either $K_{p,p}$ as a subgraph, which is impossible, or 
a proper $(\leq p)$-subdivision of $K_{p,p}$ as an induced subgraph. It is not difficult to see that in the latter case $G$ contains $S_p$ as an induced subgraph.
This final contradiction shows that the block number of $G$ is bounded by $b$. 
\end{proof}

\begin{lemma}\label{lem:final}
For every positive integer $p$, there is a positive integer $c=c(p)$ such that all $(S_p,T_p)$-free graphs with no $K_{p,p}$-subgraph
have tree-width at most $c$. 
\end{lemma}

\begin{proof}
Let $b=b(p)$ be the constant from Lemma~\ref{lem:block}, $d:=2(b-1)(b-2)$ and $G$ an $(S_p,T_p)$-free graph with no $K_{p,p}$-subgraph.  
By Lemma~\ref{lem:block}, the block number of $G$ is at most $b$, and by Theorem~\ref{thm:block}, there is a tight tree-decomposition $(T,\cal V)$
of $G$ such that every torso has fewer than $b$ vertices of degree at least $d$. Let $H$ be a torso at node $x$ of $T$.

Assume $H$ contains an induced copy of $S_p$. Then there is an edge $uv$ in this copy that does not belong to $G$ and hence
there is a bag $V_y$ with $u,v\in V_x\cap V_y$.  
Since the tree-decomposition is tight, $u$ and $v$ are connected in $G$ by a (chordless) path all of whose internal vertices belong to $U_y$ (using the notation of Section~\ref{sec:tree-width}).
We also observe that an induced $S_p$ cannot contain more than two vertices in $V_x\cap V_y$, since this is a clique. Therefore, 
by replacing every edge of $S_p$ that does not belong to $G$ with a path we obtain an induced subdivision of $S_p$ in $G$, which is impossible.
This discussion shows that $H$ is $S_p$-free. In a similar way, we can show that $H$ is $T_p$-free, because an induced copy of $T_p$ in $H$ 
gives rise either to an induced subdivision of $S_p$ in $G$ (if $uv$ is an edge of the central triangle in $T_p$) or to an induced subdivision of $T_p$
(if every edge of the central triangle in $T_p$ is also an edge in $G$).

By removing at most $b$ vertices from $H$, we obtain an $(S_p,T_p)$-free graph of vertex degree at most $d$. According to Theorem~\ref{thm:degree}
the tree-width of this graph is bounded by a constant (depending on $p$ and $d$). Since $b$ and $d$ depend only on $p$, there is constant $c=c(p)$
bounding the tree-width of each torso. 

Any tree decomposition of the torso at node $x$ has a bag which contains $V_x \cap V_y$ for each neighbour $y$ of $x$ in $T$.
Therefore, a leaf bag $V_x \cap V_y$ can be created for each $y$.
Consequently, the tree decompositions of width at most $c$ for each torso  can be combined into a tree decomposition of width at most $c$
for the whole graph $G$.

\end{proof}

\begin{theorem}
Let $X$ be a hereditary class defined by a finite set $F$ of forbidden induced subgraphs. 
There is a constant bounding the tree-width of graphs in $X$ if and only if $F$ includes a complete graph, a complete bipartite graph,
a tripod and the line graph of a tripod. 
\end{theorem}

\begin{proof}
Since $X$ is a finitely defined class, the inclusion in $F$ graphs from the four listed classes is a necessary condition for boundedness of tree-with (Lemma~\ref{lem:boundary}).
To show that this condition is sufficient, we denote by $k$ the number of connected components in a tripod in $F$ and by $\ell$ the number of connected 
components in the line graph of a tripod in $F$ and prove the theorem by induction on $k+\ell$. 

Since a complete graph and a complete bipartite graph are forbidden for $X$, graphs in $X$ contain no $K_{p,p}$-subgraph for some constant $p$. 
We choose $p$ so that every component of the tripod in $F$ is contained in $S_p$ and every component of the line graph of a tripod in $F$ is contained in $T_p$. 
In other words, we choose $p$ so that all graphs in $X$ are $(kS_p,\ell T_p)$-free with no $K_{p,p}$-subgraph. 

For $k=\ell=1$, the result follows from Lemma~\ref{lem:final}. For $k>1$, we apply the following argument. If a graph $G$ in $X$
does not contain an induced copy of $S_p$, then a bound on the tree-width of $G$ follows by induction. If $G$ contains an induced copy of $S_p$, 
then the number of vertex-disjoint such copies is bounded by $R(k,C(3p+1,p))$, because otherwise, by Lemma~\ref{lem:clique}, we obtain either $k$ copies with no edges between any two of them,
i.e. an induced  $kS_p$, or a collection of $C(3p+1,p)$ copies with an edge between any two of them, in which case a $K_{p,p}$-subgraph arises. 
Therefore, by deleting at most $(3p+1)R(k,C(3p+1,p))$ vertices from $G$ we obtain a graph, which is $S_p$-free. The tree-width of this graph is bounded by induction.
Since the number of the deleted vertices is bounded by a constant, there is a constant bounding the tree-width of $G$. 
Similar arguments apply to the case $\ell>1$.  
\end{proof}

\section{Discussion}
\label{sec:discussion}

We conclude the paper with a number of remarks and open problems. 
First, we observe that the main result of the paper, which establishes a dichotomy for hereditary classes defined by finitely many induced subgraphs,
is also suggestive towards a dichotomy in the universe of all hereditary classes. Indeed, each of the two boundary classes, $\cal S$ and $\cal T$, 
is the limit of an infinite sequence of classes of unbounded tree-width. Therefore, every class of bounded tree-width must exclude a graph from 
each class of the sequence converging to $\cal S$ and from each class of the sequence converging to $\cal T$. This is the case, for 
instance, for classes of bounded chordality, i.e. classes excluding large induced cycles $C_k,C_{k+1},\ldots$ for a constant $k$.  
The fact that bounded chordality together with bounded biclique number imply bounded tree-width was shown in \cite{absence}
by means of a probabilistic argument from \cite{KO}. Our Theorem~\ref{thm:bigclique} allows to avoid this argument, 
providing a common approach to graphs excluding induced tripods and graphs excluding large induced cycles, because an induced $(\leq p)$-subdivision 
of $K_{r,r}$ contains both an induced tripod and a large induced cycle for an appropriate value of $r$.

An interesting open question related to Theorem~\ref{thm:bigclique} is whether instead of an induced subdivision of a biclique we can
seek for an induced subdivision of a clique. More formally:

\begin{problem}
Is it true that for all positive integers $r$ and $p$, there is a positive integer $m=m(r,p)$
such that every graph $G$ containing a $(\le p)$-subdivision of $K_m$ as a subgraph 
contains either $K_{p,p}$ as a subgraph or a proper $(\leq p)$-subdivision of $K_{r}$ as an induced subgraph,
whose branch vertices are also the branch vertices of the $K_m$? 
\end{problem}

Finally, we note that a similar question with respect to Lemma~\ref{lem:clique}, i.e. whether a large biclique in this lemma can be replaced with a large clique,
has a negative answer. Indeed, let $M$ be a matching of size $m$ in a $K_{m,m}$. Clearly, every two elements (edges) of $M$ are connected by an edge, 
but the graph does not contain a large clique.

\end{document}